\documentclass[11pt,reqno]{amsart}

\usepackage{enumerate}

{\theoremstyle{plain}
 \newtheorem{theorem}{Theorem}
 \newtheorem{corollary}{Corollary}
 
\newtheorem{lemma}{Lemma}
}

\DeclareMathOperator{\re}{Re}

\newcommand{\pts}[1]{\left(#1\right)}

\newcommand{\set}[1]{\left\{#1\right\}}

\newcommand{\eps}{\varepsilon}

\begin{document}

\title[Practical numbers and the distribution of divisors]{Practical numbers and \\ the distribution of divisors}

\author{Andreas Weingartner}
\address{ 
Department of Mathematics,
351 West University Boulevard,
 Southern Utah University,
Cedar City, Utah 84720, USA}
\email{weingartner@suu.edu}

\begin{abstract} 
An integer $n$ is called practical if every $m\le n$ can be written as a sum of distinct divisors of $n$.
We show that the number of practical numbers below $x$ is asymptotic to $c x/\log x$, as conjectured by Margenstern.
We also give an asymptotic estimate for the number of integers below $x$ whose maximum ratio of consecutive divisors is at most $t$, valid uniformly for $t\ge 2$. 
\end{abstract}

\maketitle

\section{Introduction}

An integer $n\ge 1$ is called \emph{practical} if all positive integers $m\le n$ can be written as a sum of distinct divisors of $n$. Fibonacci used practical numbers in connection with 
Egyptian fractions. The term \emph{practical number} is due to Srinivasan \cite{Sri}, who 
gave a partial classification of these numbers.
Stewart \cite{Stew} and Sierpinski \cite{Sier} showed that an integer 
$n \ge 2$ with prime factorization $n=p_1^{\alpha_1} \cdots p_k^{\alpha_k}$, $p_1<p_2<\ldots < p_k$,
is practical if and only if 
$$p_j\le 1+\sigma\bigg( \prod_{1\le i \le j-1} p_i^{\alpha_i} \bigg) \qquad (1\le j \le k),$$
where $\sigma(n)$ denotes the sum of the divisors of $n$.
Here and below, the last product is understood to be $1$ when $j=1$.
In analogy with well-known conjectures about the sequence of prime numbers, 
Melfi \cite{Mel} found that every even integer is the sum of two practical numbers, and that there
are infinitely many practical numbers $n$ such that $n-2$ and $n+2$ are also practical.

Let $P(x)$ denote the number of practical numbers not exceeding $x$. Estimates for $P(x)$ were obtained
by Erd\H{o}s and Loxton \cite{EL},
Hausman and Shapiro \cite{HS}, Margenstern \cite{Mar}, Tenenbaum \cite{Ten86} and finally Saias \cite{Saias1}, who showed that there are two positive constants $c_1$ and $c_2$ such that
\begin{equation*}
c_1 \frac{x}{\log x} \le  P(x) \le c_2 \frac{x}{\log x} \qquad (x\ge 2).
\end{equation*}
Margenstern's conjecture \cite{Mar} that $P(x)$ is asymptotic to $cx/\log x$ is settled by the following result. 

\begin{theorem}\label{thmP}
There is a positive constant $c$ such that for $x\ge 3$
$$
P(x)=\frac{c x}{\log x} \left\{1+O\left(\frac{\log \log x}{\log x}\right)\right\} .
$$
\end{theorem}

Theorem \ref{thmP} is a consequence of Theorem \ref{thmB}, in which
the study of $P(x)$ is viewed as a special case of the following problem.
Let $\theta$ be a real-valued arithmetic function. 
Let $\mathcal{B}$ be the set of positive integers containing $n=1$ and all those $n \ge 2$ with  $n=p_1^{\alpha_1} \cdots p_k^{\alpha_k}$, $p_1<p_2<\ldots < p_k$, which satisfy 
\begin{equation}\label{Bdef}
p_j \le \theta\bigg( \prod_{1\le i \le j-1} p_i^{\alpha_i} \bigg) \qquad (1\le j \le k).
\end{equation}
Let $B(x)$ be the number of integers $n\le x$ in $\mathcal{B}$.
Note that if $\theta(n)=\sigma(n)+1$, then $B(x)=P(x)$.

\begin{theorem}\label{thmB}
Assume that $\theta(n)$ satisfies $\theta(1)\ge 2$ and 
$$ n\le  \theta(n) \le A n (\log 2n)^a (\log\log 3n)^b   \qquad (n\ge 1)$$
for constants $A$, $a$, $b$ with $A\ge 1$, $0\le a \le 1$. If $a<1$, then 
$$
B(x)=\frac{c_\theta x}{\log x} \Big\{1+O\big( (\log x)^{a-1} (\log\log x)^b  \big)\Big\}
\qquad (x\ge 3) .
$$
If $a=1$ and $b<-1$, then
$$
B(x)=\frac{c_\theta x}{\log x} \Big\{1+O\big( (\log\log x)^{b+1}  \big)\Big\}
\qquad (x\ge 3) .
$$
In either case, $c_\theta$ is a positive constant depending on $\theta$. The 
implied constant in the error term depends on $A$, $a$ and $b$.
\end{theorem}
Theorem \ref{thmB} is a consequence of Theorem \ref{thmBh}, which allows for more general upper bounds
on $\theta (n)$ at the expense of more technical conditions.

Theorem \ref{thmP} follows from Theorem \ref{thmB} with $(a,b)=(0,1)$, since $\theta(n)=1+\sigma(n)= O( n \log\log 3n)$.
Thompson's \cite{Thom, Thom2} weakly $\varphi$-practical numbers, which are precisely the integers corresponding to $\theta(n)=n+2$, can be estimated by Theorem \ref{thmB} with $(a,b)=(0,0)$. 

A closely related problem has to do with the ratios of consecutive divisors of an integer, a topic
that goes back at least to Erd\H{o}s \cite[Theorem 3]{Erd} in 1948.
Let $1=d_1(n)<d_2(n)< \ldots < d_{\tau(n)} = n$ denote the increasing sequence of divisors of the integer $n$. 
Tenenbaum \cite{Ten79,Ten86,Ten95} laid the foundation for our understanding of the distribution of the maximum ratio of consecutive divisors. In \cite{Ten79} he showed that, for fixed $\lambda \in [0,1]$, the set of integers $n$ which satisfy
$$
\frac{\max_{1\le i < \tau(n)} \log(d_{i+1}(n) / d_i(n))}{\log n} \le \lambda
$$
has a natural density ($=d(1/\lambda)$ in our notation below).
In \cite[Lemma 2.2]{Ten86} he found that 
\begin{equation*}
   \max_{1\le i < \tau(n)} \frac{d_{i+1}(n)}{d_i(n)}  = \frac{F(n)}{n} \qquad (n\ge  2),
\end{equation*}
where
\begin{equation*}
  F(n):=
  \begin{cases}
    1  & \quad (n=1) \\
    \max\{d\,P^-(d) : d|n, \ d>1\} & \quad (n\ge 2),
  \end{cases}
\end{equation*}
and $P^-(n)$ denotes the smallest prime factor of the integer $n\ge 2$, $P^-(1)=\infty$. 
For any integer 
$n \ge 2$ with  prime factorization $n=p_1^{\alpha_1} \cdots p_k^{\alpha_k}$, $p_1<p_2<\ldots < p_k$, 
the definition of $F(n)$ clearly implies that
\begin{equation*}
  \frac{F(n)}{n} = \frac{ \max\limits_{1\le j \le k} p_j (p_j^{\alpha_j} \cdots p_k^{\alpha_k})}{n}
  =\max_{1\le j \le k}\frac{p_j}{\prod\limits_{1\le i \le j-1} p_i^{\alpha_i}}.
\end{equation*}
Let $D(x,t)$ denote the number of positive integers $n\le x$ whose maximum ratio of consecutive divisors is at most $t$. 
We have
\begin{equation*}
\begin{split}
D(x,t) & = 1 + \#\big\{2\le n\le x : \ \max_{1\le i < \tau(n)} d_{i+1}(n)/d_i(n) \le t \big\}\\
 & = 1 + \#\big\{2\le n\le x : \ F(n)/n \le t\big\} \\
& = 1+ \#\Big\{2\le p_1^{\alpha_1} \cdots p_k^{\alpha_k}\le x : \ 
p_j \le t \prod\limits_{1\le i \le j-1} p_i^{\alpha_i} \quad (1\le j \le k) \Big\}.
\end{split}               
\end{equation*}
The last of these three expressions for $D(x,t)$ will be most useful to us. 
It shows that $D(x,t)=B(x)$ with $\theta(n)=nt$. 
Thus Theorem \ref{thmB} with $(a,b)=(0,0)$ gives an asymptotic estimate for $D(x,t)$ when $t$ is fixed. 
In the following, we allow $t$ to vary with $x$.

Improving on a result by Tenenbaum \cite{Ten86,Ten95}, Saias \cite[Theorem
1]{Saias1} showed that there exist two positive constants $c_3$
and $c_4$, such that\footnote{We replaced $\log x$ by $\log xt$ so that the estimate remains valid when $x<t$.}
\begin{equation}\label{eric}
c_3 \frac{x \log t}{\log xt} \le D(x,t)\le c_4 \frac{x\log t}{\log xt} \qquad
(x\ge 1, \ t\ge 2).
\end{equation}
Let
\begin{equation*}
  v = \frac{\log x}{\log t}.
\end{equation*}

In \cite[Theorem 1]{IDD2} we found that, for $x\ge t\ge \exp\set{(\log\log x)^{5/3+\eps}}$,
\begin{equation}\label{IDD2result}
D(x,t)=x\,d(v)\set{1+O\pts{\frac{1}{\log t}}} 
\end{equation}
where $d(v)$ is given by $d(v)=0$ for $v<0$ and \cite[Lemma 4]{IDD3}
\begin{equation}\label{dinteq}
d(v)= 1-\int_0^{\frac{v-1}{2}} \frac{d(u)}{u+1} \ \omega\left(\frac{v-u}{u+1}\right) \, \mathrm{d} u \qquad (v\ge  0).
\end{equation}
Here $\omega(u)$ denotes Buchstab's function. 
Equation \eqref{dinteq} was used in \cite[Theorem 1]{IDD3} to show that 
\begin{equation}\label{IDD3T1}
d(v)=\frac{C}{v+1}\, \left\{1+ O\left(\frac{1}{(v+1)^{2}}\right)\right\} \qquad (v \ge 0),
\end{equation}
where 
$
\displaystyle C=\frac{1}{1-e^{-\gamma}}=2.280291...,
$
and $\gamma=0.577215...$ is Euler's constant.

Theorem \ref{thmD} improves the error term in \eqref{IDD2result} and removes the lower bound on $t$, giving an asymptotic formula for $D(x,t)$ as $x\to \infty$, uniformly for $t\ge 2$.

\begin{theorem}\label{thmD}
For $x\ge 1$,  $t\ge 2$, we have
\begin{equation*}
D(x,t)=x \, \eta (t) \, d(v)  \left\{1+O\left(\frac{1}{\log 2x}\right)\right\},
\end{equation*}
where
\begin{equation}\label{eta}
0< \eta_0 \le  \eta(t) = 1 + O\left(\frac{1}{\log t}\right)
\end{equation}
for some positive constant $\eta_0$.
\end{theorem}

Combining Theorem \ref{thmD} with \eqref{IDD3T1} yields
\begin{corollary}\label{cor1}
For $x\ge t\ge 2$, we have
\begin{equation*}
D(x,t)=\frac{x \, C(t) \log t}{\log xt} \left\{1+O\left(\frac{1}{\log x}+\frac{\log^2 t}{\log^2 x}\right)\right\},
\end{equation*}
where
$$0<C_0 \le C(t):=C \eta(t)= C+ O(1/\log t)$$
for some positive constant $C_0$.
\end{corollary}
This settles a conjecture expressed below Corollary 1 of \cite{IDD3}.
Corollary \ref{cor1} clearly implies
\begin{corollary}\label{cor2}
For $x\ge t\ge 2$, we have
\begin{equation*}
D(x,t)= \frac{C  x \log t}{\log xt} \set{1+O\pts{\frac{1}{\log t} +\frac{\log^2 t}{\log^2 x}}}.
\end{equation*}
\end{corollary}
This is \cite[Corollary 1]{IDD3}, but without any restriction on $t$.

With the estimate \eqref{eta}, Theorem \ref{thmD} simplifies to
\begin{corollary}\label{cor3}
For $x\ge t\ge 2$, we have
\begin{equation*}
D(x,t)= x\,d(v)\set{1+O\pts{\frac{1}{\log t}}}.
\end{equation*}
\end{corollary}
Thus \eqref{IDD2result} holds without the restriction on $t$, confirming 
a speculation expressed below Theorem 1 of \cite{IDD2}.

To summarize, the asymptotic behavior of $D(x,t)$ is revealed in its simplest form 
by Corollary \ref{cor1} when $t$ is fixed, 
by Corollary \ref{cor2} when $t\to \infty$ but $\log t / \log x \to 0$,
and by Corollary \ref{cor3} when $\log t$ and  $\log x$ are of the same order of magnitude.

Let 
\begin{equation*}
\begin{split}
M(x,t) & := \#\Big\{2\le n\le x : \ \max_{1\le i < \tau(n)} d_{i+1}(n)/d_i(n) = t \Big\}\\
 & = \#\big\{2\le n\le x : \ F(n)/n = t\big\} 
\end{split}               
\end{equation*}
and define
\begin{equation*}
S:=\big\{p/m: \  p\  \mbox{prime}, \ m\ge 1, \ F(m)\le p \big\}.
\end{equation*}

\begin{corollary}\label{cor4}
Let $x\ge t\ge 2$. If $t \notin S$, then $M(x,t)=0$. If $t=p/m\in S$, then
\begin{equation*}
M(x,p/m)= \frac{x K(p)}{m\log x} \left\{1+O\left(\frac{p\log p}{\log x}\right)\right\},
\end{equation*}
where 
$$K(p) := (C(p)-C(p-0))\log p \asymp \frac{1}{p}\ ;$$ 
moreover
$$ M(x,p/m) \asymp \frac{x}{pm \log x} \qquad (p^{3+\varepsilon} m\le x). $$
\end{corollary}

The notation $K(p) \asymp 1/p $ means that there exist two positive constants $c_1$, $c_2$,
such that $c_1 /p \le K(p) \le c_2/p$ for all primes $p$.

We will derive Corollary \ref{cor4} from Corollary \ref{cor1} in Section \ref{SectionM}.
Corollary \ref{cor4} shows that $C(t)$ is discontinuous at every $t\in S$.
Note that $S$ contains all rational numbers of the
form $p/2^j$ where $j\ge 0$ and $p$ is a prime with $p\ge 2^{j+1}$. Hence $S$
is dense in $[2,\infty)$ by the prime number theorem.

The main tool for proving Theorems \ref{thmB} and \ref{thmD} is the functional equation in Lemma \ref{mainlemma}, 
a special case of which has already been used in \cite{IDD3} to establish \eqref{IDD3T1}. Tenenbaum \cite{Ten86, Ten95}, Saias \cite{Saias1, Saias2} and the author \cite{IDD1, IDD2} have previously employed functional equations that correspond to counting the integers in question according to their largest (or smallest) prime factor, an approach which requires an additional parameter to limit the size of the prime factors. The main advantage of Lemma \ref{mainlemma} is that it does not involve any extra parameters.

\section{Preliminary Lemmas}

Let
$$ \Phi(x,y) = \# \{ n\le x : P^-(n)>y \}  .$$
For $u\ge 1$, Buchstab's function $\omega(u)$ is defined as the
unique continuous solution to the equation
\begin{equation*}
(u\omega(u))' = \omega(u-1) \qquad (u>2)
\end{equation*}
with initial condition
\begin{equation*}
u\omega(u)=1 \qquad (1\le u \le 2).
\end{equation*}
Let $\omega(u)=0$ for $u<1$ and define $\omega$ at 1 and $\omega'$ at
1 and 2 by right-continuity.

\begin{lemma}\label{omega}
We have
\begin{enumerate}[(i)]
\item $|\omega'(u)|\le 1/\Gamma(u+1) \quad (u\ge 0)$,
\item $|\omega(u)-e^{-\gamma}| \le 1/\Gamma(u+1) \quad (u\ge 0)$.
\end{enumerate}
\end{lemma}

\begin{proof}
Part (i) is a consequence of Tenenbaum \cite[Theorems III.5.5, III.6.4]{Ten}. 
Part (ii) is \cite[Lemma 1 (iii)]{IDD3}.
\end{proof}

\begin{lemma}\label{Phi}
Let $u=\frac{\log x}{\log y}$. For $x\ge 1$, $y\ge 2$, we have
$$ \Phi(x,y)= e^\gamma x \omega(u) \prod_{p\le y} \left(1-\frac{1}{p}\right) +
O\left(\frac{y}{\log y} + \frac{x e^{-u/3}}{(\log y)^2}\right).
$$
\end{lemma}

\begin{proof}
If $x\ge 2y\ge 5$, the result follows from Tenenbaum \cite[Corollary III.6.7.6]{Ten}. If $2\le y \le 5/2$ or $y>x/2$, it is easy to verify
that the error term $O(y/\log y)$ is adequate.
\end{proof}

Let $\chi(n)$ be the characteristic function of the set $\mathcal{B}$ described in \eqref{Bdef}.
Let $P^+(n)$ denote the largest prime factor of the integer $n\ge 2$, and put $P^+(1)=1$.
The main tool of this paper is the following functional equation, which generalizes \cite[Lemma 2]{IDD3}.

\begin{lemma}\label{mainlemma}
Let $\theta(n)$ be a real-valued arithmetic function with $\theta(n)\ge P^+(n)$.
For $x\ge 0$, we have
\begin{equation*}
[x]=\sum_{n\le x} \chi(n) \, \Phi(x/n, \theta(n)) .
\end{equation*}
\end{lemma}

\begin{proof}
We show that every positive integer $m\le x$ can be written uniquely as $m=nr$, 
where $n \in \mathcal{B}$ and $P^{-}(r)>\theta(n)$. If $m=1$, we have $n=1\in \mathcal{B}$
and $P^{-}(r)=P^{-}(1)=\infty > \theta(1)$.
If $2\le m \le x$, we write $m=p_1^{\alpha_1} p_2^{\alpha_2}\cdots p_k^{\alpha_k}$, where $p_1<p_2<\ldots < p_k$.
Define $n$ to be the largest possible divisor of $m$ of the form
\begin{equation}\label{nform}
n= \prod_{1\le i \le j} p_i^{\alpha_i}   \qquad (0\le j \le k)
\end{equation}
such that $n\in \mathcal{B}$ and let $ r=m/n$.
Since $n$ is maximal, $p_{j+1}= P^{-}(r)>\theta(n)$ when $n<m$. If $n=m$, $r=1$ and $P^{-}(r)=\infty> \theta(n)$. This shows that we can write every $m\le x$ as $m=nr$ with $n \in \mathcal{B}$ and $P^{-}(r)>\theta(n)$. 

The uniqueness of the pair $(n,r)$ follows from $P^{-}(r)>\theta(n) \ge P^+(n)$,
which implies that $n$ must be of the form \eqref{nform}. Also, $n$ must be the largest divisor of $m$ of the form \eqref{nform} with $n\in \mathcal{B}$ or else $p_{j+1}\le \theta(n)$.
\end{proof}

\section{Proof of Theorem \ref{thmD}}

Throughout the rest of the paper, we write `$f(x)\ll g(x)$ for $x \in A$' or `$f(x) = O(g(x))$ for $x\in A$' to mean 
that there is a constant $c$ such that $|f(x)| \le c|g(x)|$ for all $x \in A$.
We write `$f(x) \asymp g(x)$' to mean that $f(x)\ll g(x)$ and $g(x)\ll  f(x)$.
 
Let
\begin{equation*}
\chi_t(n)= 
\begin{cases} 1 & \text{if $F(n)\le nt$,}
\\
0 &\text{else.}
\end{cases}
\end{equation*}

\begin{lemma}\label{lem0}
For $x\ge 0$, $t\ge 1$, we have
\begin{equation*}
D(x,t)=D(\sqrt{x/t},t)+[x]-\sum_{n\le \sqrt{x/t}} \chi_t(n) \, \Phi(x/n, nt) .
\end{equation*}
\end{lemma}

\begin{proof}
This follows from Lemma \ref{mainlemma} with $\theta(n)=nt$ since $\Phi(x/n, nt)=1$ when $\sqrt{x/t}<n \le x$.
\end{proof}

\begin{lemma}\label{lem1}
For $x\ge 1$, $t \ge 2$, we have
\begin{multline*}
D(x,t) = \\
x\ -\ x\sum_{n\le \sqrt{x/t}} \frac{\chi_t(n)}{n} \, e^\gamma \omega\left(\frac{\log x/n}{\log nt}\right) 
\prod_{p\le nt} \left(1-\frac{1}{p}\right) + O\left(1+\frac{x \log t}{(\log xt)^2}\right).
\end{multline*}
\end{lemma}

\begin{proof}
If $x<t$, the sum is empty and $D(x,t)=[x]=x+O(1)$. 
If $x\ge t$, we apply Lemma \ref{Phi} to estimate each occurrence of $\Phi(x/n, nt)$ in Lemma \ref{lem0}. 
The contribution from the error term $O(y/\log y)$ is
\begin{equation*}
\ll \sum_{n\le \sqrt{x/t}} \chi_t(n) \frac{nt}{\log nt}
 \ll \frac{\sqrt{xt}}{\log\sqrt{xt}} \sum_{n\le \sqrt{x/t}} \chi_t(n) 
\ll \frac{x \log t}{(\log xt)^2},
\end{equation*}
by \eqref{eric}. For the contribution from $O\left( \frac{xe^{-u/3}}{(\log y)^2}\right)$,
we can split up the interval $[1,\sqrt{x/t}]$ by powers of $2$ and use \eqref{eric} to get
\begin{equation*}
\begin{split}
 & \ll \sum_{n\le \sqrt{x/t}} \chi_t(n) \frac{x}{n (\log nt)^2} \exp\left(-\frac{\log x/n}{3\log nt}\right) \\
 & \ll \sum_{n\le \sqrt{x/t}} \frac{x \log t}{n (\log nt)^3} \exp\left(-\frac{\log xt}{6\log nt}\right) \\
  & \ll \frac{x \log t}{(\log xt)^2}.
\end{split}
\end{equation*}
\end{proof}

\begin{lemma}\label{lem2}
For $t \ge 2$, we have
\begin{equation*}
1= \sum_{n\ge 1} \frac{\chi_t(n)}{n} \prod_{p\le nt} \left(1-\frac{1}{p}\right).
\end{equation*}
\end{lemma}

\begin{proof}
Fix $t\ge 2$ and let $x\to \infty$. Lemma \ref{lem1} and \eqref{eric} imply
$$ o(1) = 1 - \sum_{n\le \sqrt{x/t}} \frac{\chi_t(n)}{n} \, e^\gamma \omega\left(\frac{\log x/n}{\log nt}\right) 
\prod_{p\le nt} \left(1-\frac{1}{p}\right) .
$$
If $\log nt \le \sqrt{\log xt}$, then 
$$\frac{\log x/n}{\log nt} = \frac{\log xt}{\log nt}-1 \ge \sqrt{\log xt}-1,$$ 
hence
$$ \left| 1-e^\gamma \omega\left(\frac{\log x/n}{\log nt}\right) \right| \ll \exp \left(-\sqrt{\log xt}\right),$$
by Lemma \ref{omega}.
Thus, the contribution to the last sum from $n$ satisfying $\log nt \le \sqrt{\log xt}$ is 
$$ o(1)+ \sum_{\log nt \le \sqrt{\log xt}} \frac{\chi_t(n)}{n} \prod_{p\le nt} \left(1-\frac{1}{p}\right).$$
The result now follows since the contribution from $n$ with $\log nt > \sqrt{\log xt}$ is $o(1)$. 
\end{proof}

\begin{lemma}\label{lem3}
For $x\ge 1$, $t \ge 2$, we have
\begin{equation*}
D(x,t)= x\sum_{n\ge 1} \frac{\chi_t(n)}{n \log nt} 
\left( e^{-\gamma} - \omega\left(\frac{\log x/n}{\log nt}\right)\right)
 + O\left(1+\frac{x \log t}{(\log xt)^2}\right).
\end{equation*}
\end{lemma}

\begin{proof}
Since $\omega(u)=0$ for $u<1$, combining Lemmas \ref{lem1} and \ref{lem2} shows that $D(x,t)$ equals
\begin{equation*}
x\sum_{n\ge 1} \frac{\chi_t(n)}{n} 
\prod_{p\le nt} \left(1-\frac{1}{p}\right)
\left(1- e^\gamma \omega\left(\frac{\log x/n}{\log nt}\right)\right)
 + O\left(1+\frac{x \log t}{(\log xt)^2}\right).
\end{equation*}
As in the proof of Lemma \ref{lem2}, the contribution from $n$ with $\log nt \le \sqrt{\log xt}$ is 
$\ll x \exp \left(-\sqrt{\log xt}\right)$. 

For those $n$ for which $\log nt > \sqrt{\log xt}$, we use the estimate
$$ \prod_{p\le nt} \left(1-\frac{1}{p}\right) = \frac{e^{-\gamma}}{\log nt} 
\left(1+O\left(\frac{1}{(\log nt)^4}\right)\right).$$
The contribution from the error term is 
$$\ll x \sum_{\log nt > \sqrt{\log xt}} \frac{1}{n (\log nt)^5} \ll \frac{x}{(\log xt)^2}.$$
\end{proof}

\begin{lemma}\label{lem4}
For $x\ge 1$, $t \ge 2$, we have
\begin{equation*}
D(x,t)= x\int_{1}^{\infty} \frac{D(y,t)}{y^2 \log yt} 
\left( e^{-\gamma} - \omega\left(\frac{\log x/y}{\log yt}\right)\right) \, \mathrm{d}y
 + O\left(1+\frac{x \log t}{(\log xt)^2}\right).
\end{equation*}
\end{lemma}

\begin{proof}
This result follows from applying partial summation to the sum in Lemma \ref{lem3}. All error terms are found to
be acceptable with the help of \eqref{eric} and Lemma \ref{omega}. We omit the calculations since they are standard. 
\end{proof}

\begin{proof}[Proof of Theorem \ref{thmD}]
From Lemma \ref{lem4} we have, for $x\ge 1$, $t \ge 2$, 
\begin{equation}\label{inteq}
D(x,t)= x \, \alpha(t) - x \int_{1}^{\infty} \frac{D(y,t)}{y^2 \log yt} 
\, \omega\left(\frac{\log x/y}{\log yt}\right)\mathrm{d}y
 + O\left(1+\frac{x \log t}{(\log xt)^2}\right),
\end{equation}
where
$$ \alpha(t):=e^{-\gamma} \int_{1}^{\infty} \frac{D(y,t)}{y^2 \log yt} \, \mathrm{d}y . $$
For $x\ge 1$, let $z\ge 0$ be given by  
$$ x=t^{e^z -1}$$
and let
$$ G_t(z) := \frac{D\left(t^{e^z -1},t\right)}{t^{e^z -1}} \, e^z = \frac{D(x,t)}{x} \frac{\log xt}{\log t} \asymp 1.$$
Multiplying \eqref{inteq} by $e^z/x$ and changing variables in the integral
via $y=t^{e^u -1}$, we get, for $z\ge 0$, $t\ge 2$,
\begin{equation}\label{conv}
\begin{split}
G_t(z) & =  \alpha(t) e^z - \int_{0}^{z} G_t(u) \, \omega\left(e^{z-u} -1\right) e^{z-u} \, \mathrm{d}u +E_t(z) \\
       & = \alpha(t) e^z - \int_{0}^{z} G_t(u) \, \Omega(z-u) e^{z-u} \, \mathrm{d}u +E_t(z), 
\end{split} 
\end{equation}
where 
\begin{equation}\label{Error}
E_t(z) \ll \frac{e^z}{t^{e^z -1}} + \frac{1}{e^z \log t} 
\end{equation}
and
$$ \Omega(u):= \omega\left(e^{u} -1\right) . $$
Now multiply \eqref{conv} by $e^{-zs}$, where $s\in \mathbb{C}$, $\re s > 1$, and integrate over $z\ge 0$ to obtain the equation of Laplace transforms
$$ \widehat{G}_t(s) = \frac{\alpha(t)}{s-1} -\widehat{G}_t(s) \, \widehat{\Omega}(s-1) + \widehat{E}_t(s) \qquad (\re s >1).$$
Hence,
$$\widehat{G}_t(s) = \frac{\alpha(t)}{(s-1) (1+\widehat{\Omega}(s-1))} + \frac{\widehat{E}_t(s)}{1+\widehat{\Omega}(s-1)}
\qquad (\re s >1).$$
Equation \eqref{dinteq} written in terms of 
$$ G(z):= e^z d(e^z -1) $$
is
$$ G(z) =e^z - \int_{0}^{z} G(u) \, \Omega(z-u) e^{z-u} \, \mathrm{d}u.$$ 
It follows that the Laplace transform of $G(z)$ is given by
$$ \widehat{G}(s) = \frac{1}{(s-1) (1+\widehat{\Omega}(s-1))} \qquad (\re s >1).$$
Thus,
\begin{equation*}
\begin{split}
\widehat{G}_t(s) & = \alpha(t) \widehat{G}(s) + \widehat{E}_t(s) \widehat{G}(s) (s-1) \\
& = \alpha(t) \widehat{G}(s) +\widehat{E}_t(s) ( \widehat{G'}(s) - \widehat{G}(s)+1 ),
\end{split}
\end{equation*}
since $G(0)=1$. Now
$$ G'(u)-G(u) = e^{2u} d'(e^u-1) = -C + O\left(e^{-2u}\right)$$
by \cite[Corollary 5]{IDD3}. Inversion of the Laplace transforms yields
\begin{equation*}
G_t(z) = \alpha(t) G(z) + \int_0^z E_t(u)\left( -C + O\left(e^{-2(z-u)}\right)\right) \mathrm{d}u + E_t(z).
\end{equation*}
From \eqref{Error} we have
\begin{equation}\label{betadef}
\beta(t) := -\int_0^\infty E_t(u) \, \mathrm{d}u 
= -\int_0^z E_t(u)\, \mathrm{d}u + O\left(\frac{1}{e^z \log t}\right)
\end{equation}
and
$$ \int_0^z E_t(u) \cdot O\left(e^{-2(z-u)}\right) \mathrm{d}u =  O\left(\frac{1}{e^z \log t}\right).$$
Thus,
\begin{equation*}
G_t(z) = \alpha(t) e^z d(e^z-1) + C \beta(t) +O\left(\frac{e^z}{t^{e^z -1}} + \frac{1}{e^z \log t}\right)
\end{equation*}
and
\begin{equation}\label{last}
D(x,t)=x \left(\alpha(t) d(v) + \frac{C \beta(t)}{v+1}\right) +O\left(1+\frac{x \log t}{(\log tx)^2}\right),
\end{equation}
for $x\ge1$,  $t \ge 2$. 
Note that \eqref{Error} and \eqref{betadef} imply $\beta(t) \ll 1/\log t $. To see that
$\alpha(t) -1 \ll 1/\log t$, put $x=t$ in \eqref{last} and use $D(x,x)=[x]$, $d(1)=1$. 
Hence \eqref{IDD3T1} allows us to write
\begin{equation}\label{lastlast}
D(x,t)=x\,  \eta(t) \, d(v)+O\left(1+\frac{x \log t}{(\log tx)^2}\right),
\end{equation}
for $x\ge1$,  $t \ge 2$, where
$$ \eta(t):=\alpha(t)+\beta(t) = 1 + O\left(\frac{1}{\log t}\right).$$
The lower bound $\eta(t)\ge \eta_0 >0$ follows for bounded $t$ from \eqref{lastlast} and \eqref{eric}.
Since $d(v)\gg 1/(v+1)$ by \eqref{eric} and \eqref{IDD2result}, the proof of Theorem \ref{thmD} is complete.
\end{proof}

\section{Proof of Corollary \ref{cor4}}\label{SectionM}

\begin{proof}
We first show that 
$$\{F(n)/n : n\ge 2\} = S .$$
Let 
$n=p_1^{\alpha_1} \cdots p_k^{\alpha_k}$ with $p_1<p_2<\ldots < p_k$. We have 
\begin{equation*}
  \frac{F(n)}{n}=\max_{1\le j \le k}\frac{p_j}{\prod\limits_{1\le i \le j-1} p_i^{\alpha_i}}
  =\frac{p_{j_0}}{\prod\limits_{1\le i \le j_0-1} p_i^{\alpha_i}}=:\frac{p}{m},
\end{equation*}
for some $j_0$, $1\le j_0 \le k$. If $j_0=1$, then $m=1$ and $F(n)/n=p\in S$. If $j_0>1$, then
\begin{equation*}
  \frac{F(m)}{m}=\max_{1\le j \le j_0-1}\frac{p_j}{\prod\limits_{1\le i \le j-1} p_i^{\alpha_i}}
  \le \frac{p}{m},
\end{equation*} 
so $F(m)\le p$ and $F(n)/n \in S$.
Conversely, if $p/m \in S$, then $m\le F(m) \le p$, so $F(mp)=\max(F(m)p,p^2)=p^2$. Thus 
$n=mp$ satisfies $F(n)/n= p/m$.

Next, we show that for $p/m \in S$ we have 
\begin{equation}\label{Mxpm}
M(x,p/m) = \#\{mpr\le x: P^-(r)\ge p, \ F(r)/r \le p^2 \}. 
\end{equation} 
From above we know that if $F(n)/n=p/m$, we must have $n=mpr$ with $ P^-(r)\ge p$.
Since $p^2r=F(n)=F(pr)=\max(p^2r, F(r))$, it follows that $F(r)\le p^2 r$.
Conversely, if $n=mpr$ satisfies $F(m)\le p$, $P^-(r)\ge p$ and $F(r)/r \le p^2$,
then $F(n) = \max(F(r),p^2 r, F(m)pr) = p^2 r$ and $F(n)/n = p/m$. 

From \eqref{Mxpm} it follows  that for $p/m \in S$ we have
\begin{equation}\label{Mxm}
M(x,p/m)= M(x/m,p),
\end{equation}
hence it suffices to estimate $M(x,p)$. Since $M(x,p)=D(x,p)-D(x,p-\varepsilon)$ 
for $0<\varepsilon < 1/x$, Corollary \ref{cor1} yields, for $(\log p)^2 \le \log x$,
\begin{equation}\label{Mxp}
M(x,p)= \frac{x\left(C(p)-C(p-\varepsilon)\right) \log p}{\log x p} 
+ O\left( \frac{x \log p}{\log^2 x} \right).
\end{equation}
Now substitute $\varepsilon_1$, $\varepsilon_2$ for $\varepsilon$, where $0<\varepsilon_1< \varepsilon_2 < 1/x$, and subtract to get
$$C(p-\varepsilon_1)-C(p-\varepsilon_2) = O(1/\log x).$$
Hence $ \lim_{\varepsilon\to 0^+} C(p-\varepsilon) =: C(p-0)$
exists by Cauchy's criterion. 
Letting $\varepsilon \to 0^+ $ in \eqref{Mxp} shows that
\begin{equation}\label{Mfinal}
M(x,p)= \frac{x K(p)}{\log x p} 
+ O\left( \frac{x \log p}{\log^2 x} \right),
\end{equation}
where $K(p)= (C(p)-C(p-0)) \log p $.

By \eqref{Mxpm} we have
$$ M(x,p) = \#\{r\le x/p: P^-(r)\ge p, \ F(r)/r \le p^2 \} \asymp \frac{x}{p \log (x/p)}, $$
for $x^{1/(3+\varepsilon)}\ge p\ge p_0$ and some suitable $p_0$, according to 
Saias \cite[Theorem 1]{Saias2} and \cite[Remark 2]{IDD1}. 
When $p< p_0$, we can iterate the functional equation in \cite[Lemma 2]{IDD1}
to show that the same estimate still holds.
With \eqref{Mxm} we get, for $p/m\in S$,
$$ M(x,p/m) \asymp \frac{x}{pm \log (x/pm)} 
\asymp \frac{x}{pm \log x} \qquad (p^{3+\varepsilon} m\le x). $$
Thus $ K(p) \asymp 1/p $ and the result follows from \eqref{Mxm} and \eqref{Mfinal}.
\end{proof}

\section{Proof of Theorem \ref{thmB}}

We will establish the following result, which is slightly more general than Theorem \ref{thmB}.
\begin{theorem}\label{thmBh}
Assume that $\theta(n)$ satisfies $\theta(1)\ge 2$ and 
$$ n\le  \theta(n) \le n f(n)  \qquad (n\ge 1)$$
for some non-decreasing function $f(x)$ for which 
$(\log f(x))^2/\log 2x$ is decreasing for sufficiently large $x$, and which satisfies
\begin{equation}\label{Hdef}
f(x) \ll \frac{\log 2x}{(\log \log 3x)^{1+\varepsilon}} \qquad (x \ge 1)
\end{equation} 
for some $\varepsilon >0$. Define
\begin{equation*}
h(x):=\int_x^\infty \frac{f(y)}{y (\log 2y)^2} \,\mathrm{d} y .
\end{equation*} 
There is a positive constant $c_\theta$ depending on $\theta$ such that
$$
B(x)=\frac{c_\theta x}{\log x} \Big\{1+O\left( h(x)\right)\Big\}
\qquad (x\ge 2) .
$$
\end{theorem}

Theorem \ref{thmB} follows from Theorem \ref{thmBh} with $f(x)=A(\log 2x)^a (\log \log 3x)^b$ for suitable constants $A$, $a$, $b$.

The proof of Theorem \ref{thmBh} is quite similar to that of Theorem \ref{thmD}.
Some extra effort is required because, unlike with $D(x,t)$, the order of magnitude of 
$B(x)$ is not known from the start. Lemma \ref{Bestimate} gives a first approximation.
Another difference is that $\log \theta(n)$ will have to be approximated by $\log 2n$ 
in Lemma \ref{lem3.5P}.  

\begin{lemma}\label{Bestimate}
Assume that $\theta(n)$ satisfies $\theta(1)\ge 2$ and $n\le \theta(n) \le n f(n)$ for $n\ge 1$,
where $f(x)$ is a non-decreasing function.
Then 
$$ \frac{x}{\log 2x} \ll B(x) \ll  \frac{x\log f(x)}{\log 2x} \qquad (x\ge 1).$$
\end{lemma}

\begin{proof}
If $n\le x/2$ is counted in $D(x/2,2)$ then \eqref{Bdef} implies that $2n$ is counted in $B(x)$,
since $\theta(1)\ge 2$ and $\theta(n)\ge n$. Thus,
$$B(x) \ge D(x/2,2) \gg \frac{x}{\log 2x}$$
by \eqref{eric}.

If $n\le x$ is counted in $B(x)$, then $n$ is also counted in $D(x,f(x))$, since
$\theta(n)\le n f(n) \le n f(x)$. Hence,
$$B(x) \le D(x,f(x)) \ll \frac{x\log f(x)}{\log 2x}$$
by \eqref{eric}.
\end{proof}

In the following, assume that $\theta(n)$ and $f(x)$ satisfy the conditions
of Theorem \ref{thmBh}, and that
\begin{equation}\label{Bbound}
B(x) \ll \frac{x g(x)}{\log 2x} \qquad (x\ge 1),
\end{equation}
for some non-decreasing function $g(x)$. Lemma \ref{Bestimate} shows that we may
assume $1\le g(x)\le \log f(x) \ll \log \log 3x$.

\begin{lemma}\label{lemP0}
For $x\ge 1$, we have 
$$ [x]=B(x)-B(\sqrt{x}) + \sum_{n\le \sqrt{x}} \chi(n) \Phi(x/n, \theta(n)).$$
\end{lemma}

\begin{proof}
This follows from Lemma \ref{mainlemma} since $\theta(n)\ge n$ and $\Phi(x/n, \theta(n))=1$ for $n>\sqrt{x}$.
\end{proof}

\begin{lemma}\label{lem1P}
For $x\ge 1$, we have
\begin{equation*}
B(x)=x\ -\ x\sum_{n\le \sqrt{x}} \frac{\chi(n)}{n} \, e^\gamma \omega\left(\frac{\log x/n}{\log \theta(n)}\right) 
\prod_{p\le \theta(n)}  \left(1-\frac{1}{p}\right) + O\left(\frac{x f(x)g(x)}{(\log 2x)^2}\right).
\end{equation*}
\end{lemma}

\begin{proof}
We apply Lemma \ref{Phi} to estimate each occurrence of $\Phi(x/n, \theta(n))$ in Lemma \ref{lemP0}. 
The contribution from the error term $O(y/\log y)$ is
\begin{equation*}
 \ll \sum_{n\le \sqrt{x}} \chi(n) \frac{n f(n)}{\log 2n}
\ll \frac{\sqrt{x} f(\sqrt{x})}{\log 2\sqrt{x}} \sum_{n\le \sqrt{x}} \chi(n) \ll \frac{x f(x)g(x)}{(\log 2x)^2}.
\end{equation*}
For the contribution from $O\left( \frac{xe^{-u/3}}{(\log y)^2}\right)$, note that $\log y=\log \theta(n) \asymp \log 2n$
and $u=\log(x/n)/\log \theta(n) \asymp \log(x/n)/\log(2n)$. We can estimate the contribution from this error term as in Lemma \ref{lem1} with $t=2$ and find that it is $\ll x g(x)/(\log 2x)^2$. 
\end{proof}

\begin{lemma}\label{lem2P}
We have
\begin{equation*}
1= \sum_{n\ge 1} \frac{\chi(n)}{n} \prod_{p\le \theta(n)} \left(1-\frac{1}{p}\right).
\end{equation*}
\end{lemma}

\begin{proof}
We follow the proof of Lemma \ref{lem2}, replacing each $xt$ by $2x$ and each $nt$ by $\theta(n)$.
\end{proof}

\begin{lemma}\label{lem3P}
For $x\ge 1$, we have
\begin{equation*}
B(x)= x\sum_{n\ge 1} \frac{\chi(n)}{n \log \theta(n)} 
\left( e^{-\gamma} - \omega\left(\frac{\log x/n}{\log \theta(n)}\right)\right)
 + O\left(\frac{x f(x) g(x)}{(\log 2x)^2}\right) .
\end{equation*}
\end{lemma}

\begin{proof}
Combining Lemmas \ref{lem1P} and \ref{lem2P} we see that $B(x)$ equals
\begin{equation*}
x\sum_{n\ge 1} \frac{\chi(n)}{n} 
\prod_{p\le \theta(n)} \left(1-\frac{1}{p}\right)
\left(1- e^\gamma \omega\left(\frac{\log x/n}{\log \theta(n)}\right)\right)
 + O\left(\frac{x f(x)g(x)}{(\log 2x)^2}\right).
\end{equation*}
As in the proof of Lemma \ref{lem3}, the contribution from $n$ with $\log 2n \le \sqrt{\log 2x}$ is 
$\ll x \exp \left(-\sqrt{\log 2x}\right)$. For $n$ with $\log 2n > \sqrt{\log 2x}$ we use a
strong form of Mertens' formula to estimate the product over primes.
\end{proof}

\begin{lemma}\label{lem3.5P}
For $x\ge 1$, we have
\begin{equation*}
B(x)= x\sum_{n\ge 1} \frac{\chi(n)}{n \log 2n} 
\left( e^{-\gamma} - \omega\left(\frac{\log x/n}{\log 2n}\right)\right)
 + O\left(\frac{x f(x)g(x)}{(\log 2x)^2}\right).
\end{equation*}
\end{lemma}

\begin{proof}
The estimate
\begin{equation}\label{logtheta}
\log \theta(n) = \log(2n)\left(1+ O\left(\frac{\log f(n)}{\log (2n)}\right)\right) \qquad (n\ge 1)
\end{equation}
applied to the first occurrence of $\log \theta(n)$ in Lemma \ref{lem3P} introduces an error of size
\begin{equation*}
\begin{split}
& \ll x \sum_{n\ge 1} \frac{\chi(n) \log f(n)}{n (\log 2n)^2} \exp\left(-\frac{\log 2x}{\log 2 n}\right)\\
& \ll x g(x)\log f(x) \sum_{n\le x}  
\frac{1 }{n (\log 2n)^3} \exp\left(-\frac{\log 2x}{\log 2 n}\right)
+ x \sum_{n> x} \frac{(\log(f(n)))^2}{n (\log 2n)^3} \\
& \ll \frac{x (\log f(x))^2}{(\log 2x)^2} + x \frac{(\log(f(x)))^2}{\log 2x} 
\sum_{n> x} \frac{1}{n (\log 2n)^2}
\ll \frac{x (\log f(x))^2}{(\log 2x)^2}
\end{split}
\end{equation*}
by \eqref{Bbound} and the fact that $(\log(f(n)))^2/\log 2n$ is decreasing for $n$ large enough.

When using \eqref{logtheta} to estimate the second occurrence of $\log \theta(n)$ in Lemma \ref{lem3P}, 
we distinguish between two cases. First, if $u_2> u_1 \ge 1$, we have
$|\omega(u_2)-\omega(u_1)|\ll (u_2-u_1)e^{-u_1}$. 
Thus, the error coming from $n$ where both $\log(x/n)/\log \theta(n)\ge 1$ and $\log(x/n)/\log (2n)\ge 1$, is
\begin{equation*}
\begin{split}
& \ll x \sum_{n\le \sqrt{x}} \frac{\chi(n) (\log f(n)) \log 2x}{n (\log 2n)^3} 
\exp\left(-\frac{\log 2x}{\log 2 n}\right) \\
& \ll x (\log f(x)) g(x) \log 2x \sum_{n\le \sqrt{x}} \frac{1}{n (\log 2n)^4} 
\exp\left(-\frac{\log 2x}{\log 2 n}\right) \\
& \ll \frac{x (\log f(x))^2}{(\log 2x)^2}.
\end{split}
\end{equation*}

Second, if $u_1<1\le u_2$, we have $|\omega(u_2)-\omega(u_1)|= \omega(u_2) \asymp 1$. 
The set of $n$, where one of $\log(x/n)/\log \theta(n)$ and $\log(x/n)/\log (2n)$ is $\ge 1$
and the other is $<1$, is contained in the interval $[\sqrt{x/f(x)},\sqrt{x}]$. The error coming from such $n$ is 
\begin{equation*}
\begin{split}
& \ll x \sum_{ \sqrt{x/f(x)} \le n\le \sqrt{x}} \frac{\chi(n)}{n \log 2n}  \\
& \ll \frac{x}{\sqrt{x/f(x)} \log(2x)}  \sum_{n\le \sqrt{x}}\chi(n) 
\ll \frac{x g(x)\sqrt{f(x)}}{(\log 2x)^2},
\end{split}
\end{equation*}
by \eqref{Bbound}.
\end{proof}

\begin{lemma}\label{lem4P}
For $x\ge 1$, we have
\begin{equation*}
B(x)= x\int_{1}^{\infty} \frac{B(y)}{y^2 \log 2y} 
\left( e^{-\gamma} - \omega\left(\frac{\log x/y}{\log 2y}\right)\right) \mathrm{d}y
 + O\left(\frac{x f(x) g(x)}{(\log 2x)^2}\right).
\end{equation*}
\end{lemma}

\begin{proof}
This follows from applying partial summation to the sum in Lemma \ref{lem3.5P}. When estimating error terms, it is convenient to split the integrals at $x$. Use $g(y)\le g(x)$ for $1\le y \le x$ and 
$g(y) \le \log f(y)$ for $y\ge x$.  All new error terms are found to be $\ll x \log(f(x)) /(\log 2x)^2$ with the help of \eqref{Bbound} and Lemma \ref{omega}. 
\end{proof}

\begin{proof}[Proof of Theorem \ref{thmBh}]
From Lemma \ref{lem4P} we have, for $x\ge 1$,
\begin{equation*}
B(x)= x \, \widetilde{\alpha} - x \int_{1}^{\infty} \frac{B(y)}{y^2 \log 2y} 
\, \omega\left(\frac{\log x/y}{\log 2y}\right)\mathrm{d}y
 + O\left(\frac{x f(x)g(x)}{(\log 2x)^2}\right),
\end{equation*}
where
$$\widetilde{\alpha}:=e^{-\gamma}  \int_{1}^{\infty} \frac{B(y)}{y^2 \log 2y}\, \mathrm{d}y.$$
For $x\ge 1$, let $z\ge 0$ be given by $x=2^{e^z-1}$ and let
$$ \widetilde{G}(z):= \frac{B\left( 2^{e^z -1} \right)}{2^{e^z -1}} e^z = \frac{B(x)}{x} \frac{\log 2x}{\log 2} \ll g(x).$$
The next part of the proof is almost identical to the proof of Theorem \ref{thmD}. Just replace every $\alpha(t)$ by $\widetilde{\alpha}$, $t$ by $2$ and $G_t(z)$ by $\widetilde{G}(z)$. 
After inversion of the Laplace transforms, we get
\begin{equation}\label{Gtilde}
\widetilde{G}(z) = \widetilde{\alpha} G(z) + \int_0^z E(u)\left( -C + O\left(e^{-2(z-u)}\right)\right) \mathrm{d}u + E(z),
\end{equation}
for $z\ge 0$, where 
\begin{equation}\label{Etilde}
E(z) \ll \frac{f(2^{e^z-1})g(2^{e^z-1})}{e^z}
\end{equation}
by Lemma \ref{lem4P}. Note that \eqref{Hdef} and \eqref{Etilde} imply that $E(z) \ll (1+z)^{-\varepsilon}$. Since $G(z) \asymp 1$, \eqref{Gtilde} yields $ \widetilde{G}(z) \ll 1 + (1+z)^{1-\varepsilon}$. Hence, $g\left( 2^{e^z -1} \right) = 1 + (1+z)^{1-\varepsilon}$ is admissible in \eqref{Bbound}, and \eqref{Etilde} now shows that $E(z) \ll (1+z)^{-2\varepsilon}$. 
Thus, $ \widetilde{G}(z) \ll 1 + (1+z)^{1-2\varepsilon}$ by \eqref{Gtilde}. 
After $\lceil 1/\varepsilon \rceil $ such iterations of \eqref{Gtilde} and \eqref{Etilde}, we eventually get $ \widetilde{G}(z) \ll 1$. Thus, $g\left( 2^{e^z -1} \right) = 1$ is admissible and 
$ E(z) \ll f(2^{e^z-1}) e^{-z}.$
We have
\begin{equation*}
 -\int_0^z E(u) \, \mathrm{d}u  = -\int_0^\infty E(u) \, \mathrm{d}u + O\left(\int_z^\infty E(u) \, \mathrm{d}u\right)
  = : \widetilde{\beta} + O\left(h(2^{e^z-1})\right).
\end{equation*}
Since $f(x)$ is non-decreasing,
\begin{equation*}
\begin{split}
\int_0^z E(u)\cdot O\left(e^{-2(z-u)} \right) \mathrm{d}u 
 & \ll  f(2^{e^z-1}) \int_0^z e^{-u} e^{-2(z-u)} \, \mathrm{d}u \\ 
 & \le f(2^{e^z-1}) \, e^{-z} \ll h(2^{e^z-1}).
\end{split}
\end{equation*}
Substituting these estimates into \eqref{Gtilde} yields 
\begin{equation*}
B(x)=x \left(\widetilde{\alpha} d(v) + \frac{C \widetilde{\beta}}{v+1}\right) 
+O\left(\frac{ x h(x) } {\log 2x}\right),
\end{equation*}
for $x\ge1$,  where $v=\log x / \log 2$. Since $h(x)\gg 1/\log 2x$, \eqref{IDD3T1} implies
\begin{equation*}
B(x)= C\left(\widetilde{\alpha} +  \widetilde{\beta}\right) \frac{x\log 2}{\log 2x } 
+O\left(\frac{ x h(x)} {\log 2x}\right).
\end{equation*}
The constant $c_\theta:=C\left(\widetilde{\alpha} +  \widetilde{\beta}\right)\log 2$ is positive since $B(x)\gg x/\log 2x$ by Lemma \ref{Bestimate} and $h(x)=o(1)$ by \eqref{Hdef}.
\end{proof}

\subsection*{Acknowledgements}
The author is grateful to Eric Saias for many valuable conversations on this subject. 
In particular, the nature of the image set of $F(n)/n$ and the order of magnitude of $M(x,p/m)$ 
came to light during one of our discussions.

\end{document}